\documentclass[12pt]{article}
\usepackage{amssymb,amsfonts,amsmath,amsthm,amsrefs,dsfont,hyperref}
\usepackage[letterpaper, portrait, margin=30mm]{geometry}

\newcommand{\1}{\mathds{1}}

\newcommand{\R}{\mathbb{R}}

\newcommand{\N}{\mathbb{N}}

\newcommand{\8}{\infty}

\newcommand{\spa}{\mathrm{span}}
\newcommand{\Ker}{\mathrm{Ker~}}

\newcommand{\Co}{\mathcal{C}}
\newcommand{\Fo}{\mathcal{F}}

\newcommand{\Int}{\mathrm{int}}

\newcounter{dummy} \numberwithin{dummy}{section}
\newtheorem{theorem}[dummy]{Theorem}
\newtheorem{lemma}[dummy]{Lemma}

\newtheorem{proposition}[dummy]{Proposition}
\newtheorem{corollary}[dummy]{Corollary}

\theoremstyle{remark}
\newtheorem{remark}[dummy]{Remark}

\begin{document}

\title{Vector lattice approach to the Riesz and Stone Representation theorems}
\author{Eugene Bilokopytov\footnote{Email address bilokopy@ualberta.ca, erz888@gmail.com.}}
\maketitle

\begin{abstract}
We present vector-lattice-theoretic proofs of Riesz Representation Theorem and Stone Representation Theorem.\medskip


MSC2020 06E15, 06F20, 46A40.
\end{abstract}

\section{Introduction}

In this note we present vector-lattice-theoretic proofs of two classical theorems of abstract analysis -- Riesz Representation Theorem and Stone Representation Theorem. In a way one should expect a possibility of such proofs, since the former theorem is a description of a special class of Banach lattices, whereas the latter theorem deals with the class of Boolean algebras, which my be viewed as a sister class to the vector lattices. In both cases we base our approach on the Kakutani Representation Theorem, which says that an Archimedean vector lattice with a unit is isomorphic to a dense sublattice of the Banach lattice $\Co\left(K\right)$ that contains $\1$, where $K$ is compact Hausdorff. We will present a truncated version of Riesz theorem and an extended version of Stone theorem.\medskip

Riesz Representation Theorem contains several statements: a signed Borel measure determines a positive functional on $\Co\left(K\right)$, there is a correspondence between regular Borel and Baire measures, a signed regular Borel measure uniquely determines a functional on $\Co\left(K\right)$, and that every positive functional on $\Co\left(K\right)$ is determined by a Baire measure. The first statement is rather simple, the last one is the hardest to prove, whereas the intermediate ones are meaningful stand-alone facts from measure theory. In this article we focus on the ``hardest part'', and in our proof most of the heavy lifting is done by the vector lattice theory.\medskip

The link from vector lattices to Boolean algebras is exemplified by the notion of the components of positive elements, as well as by the Boolean algebra of projective bands in a vector lattice. Our proof of Stone Representation Theorem reveals the link in the opposite direction and makes precise the intuition that a Boolean algebra is a ``horizontal layer'' of a vector lattice.

\section{Preliminaries + Riesz}

Everywhere in this section $K$ is a compact Hausdorff space. The indicator function of $A\subset K$ is denoted as $\1_{A}$. In particular, $\1_{K}$ is the constant function $1$. If $K$ is clear from the context we will drop the subscript. A function $f:K\to \R$ is called \emph{simple} if it only attains a finite set of values. Such function necessarily has to be a linear combination of indicators. If $f$ is continuous, it is a linear combination of indicators of \emph{clopen} (simultaneously closed and open) sets. It is easy to see that simple continuous functions form a subalgebra of $\Co\left(K\right)$. This subalgebra separates points of $K$ (and so is dense) if and only if $K$ is \emph{totally disconnected} (i.e. contains no connected sets other than singletons).\medskip

Recall that a \emph{vector lattice} is a vector space $E$ endowed with the order $\ge$ such that $E_{+}=\left\{e\in E,~ e\ge 0_{E}\right\}$ is a convex cone, and every $e,f\in E$ have a maximum $e\vee f$ and a minimum $e\wedge f$. In fact, the last pair of requirements can be replaced by existence of $\left|e\right|=e\vee -e$, for every $e\in E$. The following three properties of a vector lattice will be important to us: $E$ is called \emph{Dedekind complete}, if every $A\subset E$ such that $A\le e$, for some $e\in E$, has a supremum; $E$ is said to have \emph{principal projection property} (PPP) if $\sup\limits_{n\in\N}e\wedge nf$ exists for every $e,f\in E_{+}$; finally $E$ is called \emph{Archimedean} if $\inf\limits_{n\in\N}\frac{1}{n}e=0_{E}$, for every $e\in E_{+}$. These properties are listed in the order of descending strength.

A linear map $T:E\to F$ between vector lattices is \emph{positive} if $TE_{+}\subset F_{+}$. In particular, if $F=\R$ the functionals of the form $\nu-\mu$, where $\nu,\mu\in F'_{+}$ form a complete vector lattice. $T:E\to F$ is called a \emph{homomorphism}, if it preserves the lattice operations (equivalently, $T\left|e\right|=\left|Te\right|$, for every $e\in E$). An \emph{isomorphism} is a bijective homomorphism. An element $e\in E_{+}$ is called a \emph{(strong) unit} of $E$, if for every $f\in E$ there is $\alpha\ge 0$ such that $\left|f\right|\le \alpha e$. It turns out that Archimedean vector lattices with units are of a specific form (see \cite[Theorem 2.1.3]{mn}).

\begin{theorem}[Kakutani Representation Theorem]
If $E$ is an Archimedean vector lattice with a unit $e$, then there is a compact Hausdorff space $K$, a dense sublattice $F$ of $\Co\left(K\right)$ and a lattice isomorphism $T:E\to F$ such that $Te=\1$.
\end{theorem}

An \emph{AL space} is a vector lattice endowed with a norm $\|\cdot\|$ such that $\|e+f\|=\|e\|+\|f\|$, for every $e,f\in E_{+}$. It is easy to see that $L_{1}\left(\mu\right)$ is an AL space, for every measure $\mu$. We will use the fact that $\Co\left(K\right)^{*}$ is an AL space, and the dual of an AL space is of the form $\Co\left(K\right)$ (combine Kakutani representation theorem with \cite[Proposition 1.4.7]{mn}).\bigskip

Let $\mathcal{B}_{K}$ be the \emph{Baire} $\sigma$\emph{-algebra} of $K$, i.e. the minimal algebra of subsets of $K$ which makes every $f\in\Co\left(K\right)$ measurable. Clearly, the Baire $\sigma$-algebra is a subalgebra of the Borel $\sigma$-algebra. These subalgebras are in fact equal if $K$ is a compact metric space (see \cite[21.6 Theorem 20]{royden}). A Baire measure is a finite measure on $\mathcal{B}_{K}$. Every Baire measure on a compact space is regular (see \cite[21.6 Proposition 22]{royden}).\medskip

Recall that a subset $A$ of $K$ is called \emph{meager} (or first category) if $A=\bigcup\limits_{n\in\N}A_{n}$, where each $A_{n}$ is nowhere dense, i.e. $\Int \overline{A_n}=\varnothing$. Clearly, a subset of a meager set and a union of a countable collection of meager sets are meager, and so meager sets form a $\sigma$-ideal of subsets of $K$. Since $K$ is compact, it is a \emph{Baire space}, i.e. the only open meager subset of $K$ is $\varnothing$ (see \cite[IX.5, Theorem 1]{bourb}).

Also recall that $U\subset K$ is called \emph{regularly open} if $U=\Int\overline{U}$. We will call $V\subset K$ \emph{almost open} if $V=U\bigtriangleup A$, where $U$ is regularly open and $A$ is meager. Note that since $K$ is a Baire space, this representation is unique, and the collection $\mathcal{A}_{K}$ of almost open sets is a $\sigma$-algebra, which contains all open, and therefore all Borel sets (see \cite[theorems 4.3 and 4.6]{oxtoby}). Hence, the Boolean algebra of regularly open subsets of $K$ is the quotient of $\mathcal{A}_{K}$ with respect to the $\sigma$-ideal of meager sets. Almost open sets are often called sets with a Baire property, but we will not use this term to avoid confusion with the Baire $\sigma$-algebra.\medskip

We will use the following observation: if $\left\{V_{n}\right\}_{n\in\N}\subset \mathcal{A}_{K}$ is an increasing sequence with $V=\bigcup\limits_{n\in\N}V_{n}$, and $V_{n}=U_{n}\bigtriangleup A_{n}$, for every $n\in\N$, where $U_{n}$ is regularly open, and $A_{n}$ is meager, then $V=U\bigtriangleup A$, where $U=\Int\overline{\bigcup\limits_{n\in\N}U_{n}}$, and $A$ is meager. Indeed, $U\backslash V\subset \bigcup\limits_{n\in\N}A_{n}$ and $V\backslash U\subset \partial \left(\bigcup\limits_{n\in\N}U_{n}\right)\cup \bigcup\limits_{n\in\N}A_{n}$ are both meager. This observation essentially tells that the quotient map from $\mathcal{A}_{K}$ onto the algebra of regularly open sets is $\sigma$-order-continuous.\medskip

Finally, recall that $K$ is called \emph{extremely disconnected} if the closure of every open set in $K$ is open (and so clopen). Every regularly open subset of $K$ in this case is clopen, and so $\mathcal{A}_{K}$ consists of sets of the form $V=U\bigtriangleup A$, where $U$ is clopen and $A$ is meager. We are now ready to prove the Riesz Representation theorem.

\begin{theorem}
For every positive functional $\xi$ on $\Co\left(K\right)$ there is a Baire measure $\mu$ on $K$ such that $\xi\left(f\right)=\int f d \mu$, for every $f\in \Co\left(K\right)$.
\end{theorem}
\begin{proof}
We know that $\Co\left(K\right)^{*}$ is an AL space, and so $\Co\left(K\right)^{**}=\Co\left(L\right)$, for some compact $L$. Since $\Co\left(L\right)$ is a dual Banach lattice, it is Dedekind complete (see \cite[Theorem 1.3.2]{mn}), and so $L$ is extremely disconnected (see \cite[Proposition 2.1.4]{mn}).\medskip

Let $T$ be the natural inclusion of $\Co\left(K\right)$ into $\Co\left(K\right)^{**}=\Co\left(L\right)$, which is a homomorphism (see \cite[Proposition 1.4.5]{mn}). For every  $\nu\in \Co\left(K\right)^{*}_{+}$ we have $$\left<\1_{K},\nu\right>=\max\limits_{f\in \left[-\1_{K},\1_{K}\right]}\left<f,\nu\right>=\|\nu\|=\max\limits_{h\in \left[-\1_{L},\1_{L}\right]}\left<\nu,h\right>=\left<\nu,\1_{L}\right>,$$
and so $T\1_{K}=\1_{L}$. Therefore, $T$ is a continuous isometric lattice homomorphism from $\Co\left(K\right)$ into $\Co\left(L\right)$ such that $T\1_{K}=\1_{L}$, from where there is a continuous map $\varphi:L\to K$ such that $Tf=f\circ\varphi$, for every $f\in \Co\left(K\right)$ (see \cite[Theorem 3.2.12]{mn}).\medskip

We will define a measure $\mu_{\xi}$ on $\mathcal{A}_{L}$ so that $\xi\left(f\right)=\int fd\mu_{\xi}$, for every $f\in \Co\left(L\right)$. For every clopen $U\subset L$ and meager $A\subset L$ define $\mu_{\xi}\left(U\bigtriangleup A\right)=\xi\left(\1_{U}\right)$, where the latter is defined since $\1_{U}\in \Co\left(L\right)$ and $U$ is uniquely determined by $U\bigtriangleup A$. It is clear that $\mu_{\xi}$ is finitely additive. To show $\sigma$-additivity, let $\left\{V_{n}\right\}_{n\in\N}\subset \mathcal{A}_{L}$ be an increasing sequence and let $V=\bigcup\limits_{n\in\N}V_{n}$; let $\left\{U_{n}\right\}_{n\in\N}$ and $U$ be the corresponding clopen sets. From the comments before the theorem, $\left\{U_{n}\right\}_{n\in\N}$ is increasing and $U=\Int\overline{\bigcup\limits_{n\in\N}U_{n}}$, from where  $\1_{U_{n}}\uparrow\1_{U}$. Since $\xi$ is an element of the pre-dual of $\Co\left(L\right)$, it is order continuous (see \cite[Theorem 1.4.14]{mn}), and so $\xi\left(\1_{U_{n}}\right)\uparrow\xi\left(\1_{U}\right)$. Hence, $\mu_{\xi}\left(V_{n}\right)=\xi\left(\1_{U_{n}}\right)\uparrow \xi\left(\1_{U}\right)=\mu_{\xi}\left(V\right)$.\medskip

We get that both $\xi$ and $f\to \int fd\mu_{\xi}$ are continuous linear functionals on $\Co\left(L\right)$, which agree on the subalgebra of simple continuous functions. Since $L$ is extremely disconnected, it is totally disconnected, and so simple continuous functions form a dense subalgebra of $\Co\left(L\right)$, and so $\xi\left(f\right)=\int fd\mu_{\xi}$, for every $f\in \Co\left(L\right)$.

Finally, define a Baire measure $\mu$ on $K$ by $\mu\left(A\right)=\mu_{\xi}\left(\varphi^{-1}\left(A\right)\right)$, for $A\in \mathcal{B}_{K}$. Note that $A$ is Baire, hence Borel, from where $\varphi^{-1}\left(A\right)$ is Borel, hence almost open. Then $\xi\left(f\right)=\left<\xi,Tf\right>=\left<\xi,f\circ\varphi\right>=\int f\circ\varphi d \mu_{\xi}=\int f d \mu$, for every $f\in \Co\left(K\right)$.
\end{proof}

Let us recap the proof of $\sigma$-additivity in a conceptual way. First $\mathcal{B}_{K}$ are mapped into $\mathcal{A}_{L}$ via $A\to\varphi^{-1}\left(A\right)$. This correspondence is order-continuous, in the sense that it respects arbitrary unions and intersections. Then $\mathcal{A}_{L}$, is factorized with respect to the $\sigma$-ideal of the meager sets. The quotient map is $\sigma$-order-continuous. The corresponding quotient is the Boolean algebra of the clopen subsets of $L$. This algebra embeds into $\Co\left(L\right)$ via $U\to \1_{U}$ (again order-continuously). Finally, $\xi$ acts on the indicators in order continuous fashion. Thus, $\mu$ is $\sigma$-order continuous on $\mathcal{B}_{K}$ as a composition of several $\sigma$-order continuous maps.

\begin{remark}
In a similar way, for any finite measure $\mu$, one can show that there is a compact space $K$, such that $\mu$ is isomorphic to a measure on $K$, and moreover $L_{1}\left(\mu\right)^{*}=L_{\8}\left(\mu\right)=\Co\left(K\right)$.
\qed\end{remark}

We conclude this section with some additional facts about $\Co\left(K\right)$, or more precisely its dense sublattices, whose importance is justified by Kakutani representation theorem.

\begin{proposition}[Sublattice Urysohn lemma]
Let $E$ be a dense sublattice of $\Co\left(K\right)$ that contains $\1$, let $U\subset K$ be open and let $L\subset U$ be closed. Then $E_{+}$ contains a function whose value is $1$ on $L$ and $0$ outside $U$. In particular, $E$ contains all simple continuous functions.
\end{proposition}
\begin{proof}
From Tietze-Urysohn theorem there is $g\in \Co\left(K\right)$ such that $\1_{L}\le g \le \1_{U}$. Since $E$ is dense, there is $h\in E$ such that $\|h-g\|\le\frac{1}{3}$. The latter is equivalent to $-\frac{1}{3}\1\le h-g\le \frac{1}{3}\1$, or $3g-2 \1\le 3 h-\1 \le 3g$. Let $f= \left(3 h-\1\right)^{+}\wedge \1$. Then $f\le 3 g^{+}\wedge \1\le 3 \1_{U}\wedge \1=\1_{U}$, and simultaneously $f\ge \left(3g-2 \1\right)^{+}\wedge \1\ge \left(3\1_{L}-2 \1\right)^{+}\wedge \1=\1_{L}\wedge \1=\1_{L}$.

If $L$ is a clopen set, then applying the first claim to $U=L$ shows that $\1_{L}\in E$. Since any simple continuous function is a linear combination of indicators of clopen sets, it follows that $E$ contains all simple continuous functions.
\end{proof}

As was mentioned above, $K$ is extremely disconnected if and only if $\Co\left(K\right)$ is Dedekind complete. Let us characterize (compact) totally disconnected spaces.

\begin{proposition}\label{tot}
$K$ is totally disconnected if and only if $\Co\left(K\right)$ has a dense PPP sublattice that contains $\1$.
\end{proposition}
\begin{proof}
If $K$ is totally disconnected, then the set of simple continuous functions is dense in $\Co\left(K\right)$, and is also a PPP lattice (this can be easily derived from definition, or from Theorem \ref{main} below).

Necessity: Let $x,y\in K$ be distinct. Let $U$ be a neighborhood of $x$ such that $y\not\in\overline{U}$. From Sublattice Urysohn lemma there is $f\in E_{+}$ that vanishes outside $U$, and such that $f\left(x\right)\ne 0$. It is easy to see that if $\sup\limits_{n\in\N}nf\wedge \1$ exists, it has to be equal $\1_{K\backslash \Int f^{-1}\left(0\right)}$, from where $\Int f^{-1}\left(0\right)$ is clopen. Since $f\left(x\right)=1$, we have $x\not\in f^{-1}\left(0\right)\supset \Int f^{-1}\left(0\right)$. On the other hand, since $f$ vanishes outside of $U$, we have $y\in X\backslash\overline{U}\subset \Int f^{-1}\left(0\right)$. Hence, we found a clopen set that contains exactly one of an arbitrary pair of points in $X$. Thus, $X$ is totally disconnected.
\end{proof}

\section{Vector lattice over a Boolean algebra + Stone}

In this section $A$ is a Boolean algebra with the least element $\mathds{O}$ and the greatest element $\mathds{E}$.

Let us start with an observation that many of the identities that hold in vector lattices (such as \cite[Theorem 1.1.1]{mn}) in fact ``conditionally'' hold in ordered vector spaces. For example, if $E$ is an ordered vector space and $e,f\in E$ are such that $e\wedge f$ exists, then $e\vee f$ also exists and $e+f=e\wedge f + e\vee f$.

\begin{lemma}\label{homo}
Let $E$ be an ordered vector space and let $\varphi:A\to E_{+}$. Then $\varphi$ is a lattice homomorphism (preserves $\vee$ and $\wedge$) with $\varphi\left(\mathds{O}\right)=0_{E}$ if and only if $\varphi\left(a\right)\bot\varphi\left(b\right)$ and $\varphi\left(a\vee b\right)=\varphi\left(a\right)+\varphi\left(b\right)$, for any disjoint $a,b\in A$. Moreover, in this case  $\varphi\left(A\right)$ is a Boolean algebra.
\end{lemma}
\begin{proof}
We only need to prove sufficiency. For any $a,b\in A$ the elements $a\wedge b$, $a\wedge \overline{b}$ and $b\wedge \overline{a}$ disjoint, and so $\varphi\left(a\wedge \overline{b}\right)\wedge \varphi\left(b\wedge \overline{a}\right)=0_{E}$, $\varphi\left(a\wedge \overline{b}\right)+\varphi\left(a\wedge b\right)=\varphi\left(a\right)$ and $\varphi\left(\overline{a}\wedge b\right)+\varphi\left(a\wedge b\right)=\varphi\left(b\right)$. Hence,
\begin{align*}
\varphi\left(a\wedge b\right)&=\varphi\left(a\wedge b\right)+\varphi\left(a\wedge \overline{b}\right)\wedge \varphi\left(b\wedge \overline{a}\right)\\
&=\left(\varphi\left(a\wedge \overline{b}\right)+\varphi\left(a\wedge b\right)\right)\wedge \left(\varphi\left(b\wedge \overline{a}\right)+\varphi\left(a\wedge b\right)\right)=\varphi\left(a\right)\wedge\varphi\left(b\right).
\end{align*}
The equality $\varphi\left(a\vee b\right)=\varphi\left(a\right)\vee\varphi\left(b\right)$ is proven similarly, but starting from $\varphi\left(a\vee b\right)=\varphi\left(a\wedge b\right)+\varphi\left(a\wedge \overline{b}\right)\vee\varphi\left(b\wedge \overline{a}\right)$.
\end{proof}

A useful property of Boolean algebras is that if $a_{1},...,a_{n}\in A$, then there is a disjoint finite collection $B\subset A$ such that $a_{k}=\bigvee\left\{b\in B,~b\le a_{k}\right\}$, for every $k\in\overline{1,n}$. Indeed, one can take $B=\left\{\bigwedge\limits_{i=1}^{n}a_{i}^{\varepsilon_{i}},~\varepsilon_{i}=\pm 1\right\}$, where $a^{1}=a$ and $a^{-1}=\overline{a}$.

\begin{theorem}\label{main}There is a PPP vector lattice $\Fo\left(A\right)$ with the following properties:
\item[(i)] There is an injective lattice homomorphism $\varphi$ from $A$ into $\Fo\left(A\right)_{+}$ such that $\varphi\left(\mathds{O}\right)=0_{\Fo\left(A\right)}$, $\varphi\left(\mathds{E}\right)$ is a strong unit and $\spa~ \varphi\left(A\right)= \Fo\left(A\right)$.
\item[(ii)] If $F$ is an ordered vector space and $\psi:A\to F_{+}$ is a lattice homomorphism with $\psi\left(\mathds{O}\right)=0_{E}$, then $\spa~ \psi\left(A\right)$ is a lattice in the induced order, and there is a vector lattice homomorphism $J_{\psi}$ from $\Fo\left(A\right)$ onto $\spa~ \psi\left(A\right)$ such that $\psi=J_{\psi}\circ \varphi$. It is an injection if and only if $\psi$ is.
\item[(iii)] $\Fo\left(A\right)$ is isomorphic to the lattice of simple continuous functions on a compact totally disconnected space $K$ such that $A$ is isomorphic to the Boolean algebra of clopen subsets of $K$.
\end{theorem}

Of course part (iii) implies Stone theorem. We will prove the theorem in several steps. Let  $E=\bigoplus\limits_{a\in A}\R$, which is a vector space and let
$$C=\left\{\sum \alpha_{k}a_{k},~\forall a>\mathds{O}~\exists b\in\left(\mathds{O},a\right]~\forall c\in\left(\mathds{O},b\right]:~ \sum\limits_{c\le a_{k}}\alpha_{k}\ge 0\right\}.$$

\begin{lemma}\label{cone}
$C$ is the minimal convex cone in $E$ that contains $1\cdot a$, for every $a\in A$, as well as $1\cdot b +1\cdot d- 1\cdot b\vee d$ and $1\cdot b\vee d-1\cdot b-1\cdot d$, for any disjoint $b,d$.
\end{lemma}

Note that the conditions imply that $C$ contains $\sum\limits_{i=1}^{n}\alpha_{n}a_{n}$, for any $a_{1},...,a_{n}\in A$ and $\alpha_{1},...,\alpha_{n}\ge 0$, $\alpha\mathds{O}$, for any $\alpha\in\R$, as well as $1\cdot \bigvee\limits_{i=1}^{n}b_{i}-\sum\limits_{i=1}^{n}1\cdot b_{n}$ and $\sum\limits_{i=1}^{n}1\cdot b_{n}-1\cdot \bigvee\limits_{i=1}^{n}b_{i}$, for any disjoint $b_{1},...,b_{n}\in A$.

\begin{proof}
Let us first prove that $C$ is a convex cone. It is clear that $\lambda C\subset C$, for every $\lambda\ge 0$. We need to show that if $e=\sum \alpha_{k}a_{k}$ and $f=\sum \beta_{k}a_{k}$ belong to $C$, then $e+f\in C$. Let $a>\mathds{O}$. There is $b\in\left(\mathds{O},a\right]$ such that $\sum\limits_{c\le a_{k}}\alpha_{k}\ge 0$, for all $c\in\left(\mathds{O},b\right]$; there is also $d\in\left(\mathds{O},b\right]\subset \left(\mathds{O},a\right]$ such that $\sum\limits_{c\le a_{k}}\beta_{k}\ge 0$, for all $c\in\left(\mathds{O},d\right]\subset \left(\mathds{O},b\right]$. Hence, if $c\in\left(\mathds{O},d\right]$, then $\sum\limits_{c\le a_{k}}\alpha_{k}+\beta_{k}\ge 0$. Since $a$ was arbitrary, we conclude that $e+f\in C$.\medskip

Let us show that $C$ contains $1\cdot b +1\cdot d- 1\cdot b\vee d$, for disjoint $b,d$ (other proofs are similar). Let $a>\mathds{O}$. If $a$ is disjoint with both $b$ and $d$, then for any $c\in\left(\mathds{O},a\right]$ we have $b,d,b\vee d \not\ge c$, and so the sum of the coefficients at the elements that exceed $c$ is $0$. If $a$ is not disjoint with say $b$, then any $c\in\left(\mathds{O},a\wedge b\right]$ is disjoint with $d$, but $b,b\vee d \ge c$, and so the sum of the coefficients is again $0$. \medskip

To show minimality let $D$ be another cone with the declared properties and let  $e=\sum\limits_{k=1}^{n} \alpha_{k}a_{k}\in C$. Let $B\subset A$ be a finite disjoint collection such that $a_{k}=\bigvee\left\{b\in B,~b\le a_{k}\right\}$, for every $k\in\overline{1,n}$. Then, $e_{k}=\alpha_{k}\left( a_{k}-\sum\limits_{b\le a_{k}} b\right) \in D$, for every $k\in\overline{1,n}$, and $$e=\sum\limits_{k=1}^{n} e_{k}+\sum\limits_{k=1}^{n}\sum\limits_{b\le a_{k}} \alpha_{k}b=\sum\limits_{k=1}^{n} e_{k}+\sum\limits_{b\in B}\sum\limits_{b\le a_{k}} \alpha_{k}b.$$ For every $b\in B$ there is $c\in\left(\mathds{O},b\right]$ such that $\sum\limits_{c\le a_{k}} \alpha_{k}\ge 0$. But since $B$ is a disjoint collection, for every $k\in\overline{1,n}$ either $a_{k}$ is disjoint with $b$ or $a_{k}\ge b$, and so $b\le a_{k}$ if and only if $c\le a_{k}$. Hence, $\beta_{b}=\sum\limits_{b\le a_{k}} \alpha_{k}\ge 0$, for every $b\in B$, and so $e=\sum\limits_{k=1}^{n} e_{k}+\sum\limits_{b\in B} \beta_{b}b$, where all summands belong to $D$, and so $e\in D$. Thus, $C\subset D$.
\end{proof}

\begin{remark}
In fact $C$ allows a simpler definition, which is slightly more convenient for proving minimality, but much less convenient to prove that $C$ is a convex cone. Namely the part $\forall c\in\left(\mathds{O},b\right]$ can be dropped, and $b$ can be used instead of $c$. To show the equivalence assume that there is $e=\sum\limits_{i=1}^{m}\alpha_{i}a_{i}$ that satisfies the weaker definition and $a>0$ is such that for every $b\in\left(\mathds{O},a\right]$ there is $c\in\left(\mathds{O},b\right]$ such that $\sum\limits_{c\le a_{k}}\alpha_{k}< 0$. Let $b_{0}=b$ and $b_{1}=c$. If $b_{0}>...>b_{2n-1}>\mathds{O}$ are selected, there are $b_{2n}\in\left(\mathds{O},b_{2n-1}\right]$ and $b_{2n+1}\in\left(\mathds{O},b_{2n}\right]$ such that $\sum\limits_{b_{2n}\le a_{k}}\alpha_{k}\ge 0$ and $\sum\limits_{b_{2n+1}\le a_{k}}\alpha_{k}<0$. Let $A_{n}=\left\{k,~ b_{n}\le a_{k}\right\}$. Then $A_{n}$ is an increasing sequence, and since $\sum\limits_{A_{n}}\alpha_{k}$ and $\sum\limits_{A_{n+1}}\alpha_{k}$ have different signs, $A_{n}$ is strictly increasing. However, this contradicts the fact that $\left|A_{n}\right|\le m$, for every $n\in\N$.\smallskip
\qed\end{remark}

Now that we have established that $C$ is a cone, we can introduce the equivalence relation on $E$. Namely, $e\sim f$ if $e-f \in C\cap -C$. Note that if $b_{1},...,b_{n}\in A$ are disjoint, then $1\cdot b_{1}+...+1\cdot b_{n} \sim 1\cdot \bigvee_{i=1}^{n}b_{i}$. It is easy to see now that any $e_{1},...,e_{n}\in E$ admit a common disjoint representation, i.e. there are disjoint $a_{1},...,a_{m}$ and $\left\{\alpha_{i,j}\right\}_{i,j=1}^{m,n}\subset\R$ such that $e_{k}\sim\sum \alpha_{i,k}a_{i}$, for every $k$. Moreover, if $e\sim\sum \alpha_{i}a_{i}$ is a disjoint representation, then $e\in C$ if and only if $\alpha_{i}\ge 0$, for every $i$. Indeed, if $e\in C$, then for every $i$ there is $b\in\left(\mathds{O},a_{i}\right]$ such that $\sum\limits_{b\le a_{k}}\alpha_{k}\ge 0$, but since $\left\{a_{k}\right\}$ are disjoint, the only $k$ such that $a_{k}\ge b$ is $k=i$, and so $\alpha_{i}\ge 0$.\medskip

Let $\Fo\left(A\right)=E\slash\sim$, with the quotient map $Q$. Then $QC$ is a convex cone in $\Fo\left(A\right)$ that satisfies $QC\cap -QC=\left\{0_{\Fo\left(A\right)}\right\}$, and so it turns $\Fo\left(A\right)$ into an ordered vector space. Observe also that $Q^{-1}QC=C+\Ker Q= C+ C\cap -C=C$, since $C+C=C$.\medskip

\begin{lemma}\label{lattice}
$\Fo\left(A\right)$ is a vector lattice with PPP.
\end{lemma}
\begin{proof}
To prove that $\Fo\left(A\right)$ is a lattice it is enough to show that if $f\in \Fo\left(A\right)$, then there is $\left|f\right|=f\vee -f$. Let $f$ be given in its disjoint representation, i.e. $f=Qe$, where $e=\sum\limits_{i=1}^{n} \alpha_{i}a_{i}$ and $\left\{a_{i}\right\}_{i=1}^{n}$ are disjoint. Define $g\in E$ by $g=\sum\limits_{i=1}^{n} \left|\alpha_{i}\right|a_{i}$.

Let us show that $Qg$ does not depend on the disjoint representation, i.e. if $\left\{a'_{i}\right\}_{i=1}^{m}$ are disjoint, $e'=\sum\limits_{i=1}^{m} \alpha'_{i}a'_{i}\sim e$, then the analogously defined $g'$ is equivalent to $g$. Indeed, let $B\subset A$ be a finite disjoint collection such that $a_{k}=\bigvee\left\{b\le a_{k}\right\}$, for every $k\in\overline{1,n}$, and $a'_{k}=\bigvee\left\{b\le a'_{k}\right\}$, for every $k\in\overline{1,m}$. Since $\left\{a_{i}\right\}_{i=1}^{n}$ are disjoint, for every $b\in B$ there is at most one $i$ such that $b\le a_{i}$. Define $\alpha\left(b\right)=\alpha_{i}$ if such $i$ exists, and $\alpha\left(b\right)=0$ otherwise. Define $\alpha':B\to\R$ in a similar way for $e'$. Then, $e=\sum\limits_{i=1}^{n} \alpha_{i}a_{i}\sim \sum\limits_{i=1}^{n}\sum\limits_{b\le a_{i}}\alpha_{i}b=\sum\limits_{b\in B}\alpha\left(b\right)b$, while $e'\sim \sum\limits_{b\in B}\alpha'\left(b\right)b$. Since $e\sim e'$, and $B$ is a disjoint set it follows that $\alpha\left(b\right)=\alpha'\left(b\right)$, for every $b\in B$. Hence $g\sim \sum\limits_{b\in B}\left|\alpha\left(b\right)\right|b=\sum\limits_{b\in B}\left|\alpha'\left(b\right)\right|b\sim g'$.

We will now show that $Qg=f\vee -f$. First, clearly, $g-e,g+e\in C$, from where $Qg\ge f,-f$. Let $h\in \Fo\left(A\right)$ be such that $f,-f\le h$. Since $Qg$ does not depend on the disjoint representation, we may assume that $h$ has the same one as $f$, i.e. there is $w=\sum \gamma_{i}a_{i}$, such that $Qw=h$. As $Q\left(w-e\right)\in QC$, it follows that $w-e\in C$, and so $\gamma_{i}\ge \alpha_{i}$, for every $i$. Analogously, $\gamma_{i}\ge -\alpha_{i}$, for every $i$, and so $\gamma_{i}\ge \left|\alpha_{i}\right|$, for every $i$. Thus, $w-g\in C$, from where $h\ge Qg$.\medskip

In order to prove that $\Fo\left(A\right)$ has PPP we need to show that $\sup\limits_{n\in\N} f\wedge nh$ exists, for any $f,h\in\Fo\left(A\right)_{+}$. Let $f$ and $h$ be given in their common disjoint representation, i.e. $f=Qe$ and $h=Qg$, where $e=\sum \alpha_{i}a_{i}$ and $g=\sum \beta_{i}a_{i}$, $\left\{a_{i}\right\}$ are disjoint, and $\alpha_{i}, \beta_{i}\ge 0$. It follows from the proof of the previous claim that $f\wedge nh$, where $u_{n}=\sum \alpha_{i}\wedge n \beta_{i} a_{i}$, and $n\in\N$. Define $u\in E$ by $u=\sum\limits_{\beta_{i}>0} \alpha_{i}a_{i}$. In a similar way as above, one can show that $Qu$ is independent of the disjoint representation. Clearly, $u-u_{n}\in C$, and $Qu\ge Qu_{n}=f \wedge nh$. Let $w\in \Fo\left(A\right)$ be such that $f \wedge nh\le w$, for every $n\in \N$. Again, we may assume that $w$ has the same disjoint representation as $f$ and $h$, i.e. there is $v=\sum \gamma_{i}a_{i}$, such that $Qv=w$. It follows that $v-u_{n}\in C$, from where $\gamma_{i}\ge \alpha_{i}\wedge n \beta_{i}$, for every $n\in\N$ and $i$, and so if $\beta_{i}>0$, then $\gamma_{i}\ge \alpha_{i}$. Thus, $v-u\in C$, and so $w\ge Qu$.\end{proof}

\begin{proof}[Proof of parts (i)-(iii) od Theorem \ref{main}]
(i): Define  $\varphi: A\to \Fo\left(A\right)$ by $\varphi\left(a\right)=Q \left(1\cdot a\right)$. It is easy to see that $\varphi\left(\mathds{O}\right)=0_{\Fo\left(A\right)}$, $\varphi\left(\mathds{E}\right)$ is a strong unit and $\spa~ \varphi\left(A\right)= \Fo\left(A\right)$. Let us show that $\varphi$ is a lattice homomorphism. Fix disjoint $a,b\in A$. From Lemma \ref{cone} we have that $1\cdot a\vee b\sim  1\cdot a+ 1\cdot b$, from where $\varphi\left(a\vee b\right)=\varphi\left(a\right)+\varphi\left(b\right)$. Applying the reasoning from the proof of Lemma \ref{lattice}, we get
$$Q \left(1\cdot a\right)\wedge Q \left(1\cdot b\right)=Q\left(1\cdot a+0\cdot b\right)\wedge Q\left(0\cdot a+1\cdot b\right)=Q\left(0\cdot a+0\cdot b\right)=0_{\Fo\left(A\right)},$$
from where  $\varphi\left(a\right)\bot\varphi\left(b\right)$. Since $a,b$ were chosen arbitrarily, from Lemma \ref{homo} it follows that $\varphi$ is a homomorphism.

Finally, $\varphi$ is injective since if $\varphi\left(a\right)=\varphi\left(b\right)$, then $1\cdot a-1\cdot b\sim1\cdot a\wedge \overline{b} -1\cdot  b \wedge\overline{a} \in C\cap -C$, from where $a\wedge \overline{b} =\mathds{O}= b \wedge\overline{a}$, and so $a=b$.\medskip

(ii): Consider a map $T: E\to F$ defined by $T\sum \alpha_{k}a_{k}=\sum \alpha_{k}\psi\left(a_{k}\right)$. Clearly, $T \left(1\cdot a\right)=\psi\left(a\right)\ge 0_{F}$, for any $a\in A$, while from Lemma \ref{homo} it follows that\linebreak $T\left( 1\cdot b +1\cdot d- 1\cdot b\vee d\right)=\psi\left(b\right)+\psi\left(d\right)-\psi\left(b\vee d\right)=0_{F}$, for any disjoint $b,d\in A$. Hence, from Lemma \ref{cone} $TC\subset F_{+}$, from where $T\left(-C\right)\subset -F_{+}$, and so $Tf=Tg$, if $f\sim g$. Thus, $T$ induces an operator $J_{\psi}:\Fo\left(A\right)\to F$. Let us show that $J_{\psi}$ is a homomorphism, and consequently its image is a lattice. If $e=\sum \alpha_{i}a_{i}$ and $g=\sum \beta_{i}a_{i}$ is a common disjoint representation of $f,h\in \Fo\left(A\right)$, then $f\wedge h= Qu$, where $u=\sum \left(\alpha_{i}\wedge\beta_{i}\right)a_{i}$. Since $\psi$ preserves disjointness, $\left\{\psi\left(a_{i}\right)\right\}$ are disjoint, and so
\begin{align*}
J_{\psi}f\wedge J_{\psi}h&=Te\wedge Tg=\sum \alpha_{i}\psi\left(a_{i}\right)\wedge \sum \beta_{i}\psi\left(a_{i}\right)=\bigvee \alpha_{i}\psi\left(a_{i}\right)\wedge \bigvee \beta_{i}\psi\left(a_{i}\right)\\
&=\bigvee\left(\alpha_{i}\psi\left(a_{i}\right)\wedge \beta_{j}\psi\left(a_{j}\right)\right)=\sum \left(\alpha_{i}\wedge\beta_{i}\right)\psi\left(a_{i}\right)=Tu=J_{\psi}\left(f\wedge h\right).
\end{align*}
Thus, $J_{\psi}$ is a homomorphism. If $J_{\psi}$ is an injection, then so is $\psi=J_{\psi}\circ\varphi$. If $\psi$ is an injection, and $T\sum \alpha_{k}a_{k}=\sum \alpha_{k}\psi\left(a_{k}\right)=0$, where $\left\{a_{k}\right\}$ are disjoint,  $\left\{\psi\left(a_{k}\right)\right\}$ are also disjoint, and so linearly independent, from where $\alpha_{k}=0$, for every $k$. Hence, $J_{\psi}$ is an injection.\medskip

(iii): Since $\Fo\left(A\right)$ is an Archimedean vector lattice with a strong unit, there is a compact space $K$ and a lattice isomorphism $J$ from $\Fo\left(A\right)$ onto a dense sublattice of $\Co\left(K\right)$ that contains $\1$. Since $\Fo\left(A\right)$ has PPP, it follows from Proposition \ref{tot} that $K$ is totally disconnected. From Urysohn's lemma, $J\Fo\left(A\right)$ contains all simple functions. Conversely, since $a$ and $\overline{a}$ are positive, disjoint, and $a\vee \overline{a}=e$, it follows that $J\varphi\left(a\right)$ and $J\varphi\left(\overline{a}\right)$ are positive, disjoint and add up to $\1$. Hence, both of them are characteristic functions of some clopen sets in $K$. Since  $\spa~ \varphi\left(A\right)=\Fo\left(A\right)$, it follows that $J\Fo\left(A\right)$ consists of simple functions.
\end{proof}

Note that most of the material in this section is applicable to Boolean rings, which are not necessarily algebras. In particular we would like to state the following version of the theorem for rings.

\begin{corollary}If  $E$ is an ordered vector space, $0_{E}\in R\subset E_{+}$ is a Boolean ring with respect to the order of $E$, then $\spa R$ is a PPP sublattice of $E$ isomorphic to $\Fo\left(R\right)$, where the latter is defined analogously to $\Fo\left(A\right)$.
\end{corollary}


\begin{bibsection}
\begin{biblist}

\bib{bourb}{book}{
   author={Bourbaki, Nicolas},
   title={General topology. Chapters 5--10},
   series={Elements of Mathematics (Berlin)},
   note={Translated from the French;
   Reprint of the 1989 English translation},
   publisher={Springer-Verlag, Berlin},
   date={1998},
   pages={iv+363},
}

\bib{mn}{book}{
   author={Meyer-Nieberg, Peter},
   title={Banach lattices},
   series={Universitext},
   publisher={Springer-Verlag, Berlin},
   date={1991},
   pages={xvi+395},
}

\bib{oxtoby}{book}{
   author={Oxtoby, John C.},
   title={Measure and category},
   series={Graduate Texts in Mathematics},
   volume={2},
   edition={2},
   note={A survey of the analogies between topological and measure spaces},
   publisher={Springer-Verlag, New York-Berlin},
   date={1980},
   pages={x+106},
}

\bib{royden}{book}{
   author={Royden, H. L.},
   author={Patrick Fitzpatrick},
   title={Real analysis},
   edition={4},
   publisher={Prentice Hall},
   date={2010},
   pages={x+505},
}

\end{biblist}
\end{bibsection}
\end{document}